\documentclass[12pt, reqno]{amsart}
\usepackage{amsmath, amsthm, amscd, amsfonts, amssymb, graphicx, color}
\usepackage[bookmarksnumbered, colorlinks, plainpages]{hyperref}

\makeatletter \oddsidemargin.9375in \evensidemargin \oddsidemargin
\newtheorem{theorem}{Theorem}[section]
\newtheorem{lemma}[theorem]{Lemma}
\newtheorem{proposition}[theorem]{Proposition}
\newtheorem{corollary}[theorem]{Corollary}
\theoremstyle{definition}
\newtheorem{definition}[theorem]{Definition}
\newtheorem{example}[theorem]{Example}

\theoremstyle{remark}
\newtheorem{remark}[theorem]{Remark}
\numberwithin{equation}{section}

\begin{document}
\setcounter{page}{1}

%%%%%%%%%%%%%%%%%%%%%%%%%%%%%%%%%%%%%%%%%%%%%%%
%% Please do not remove the following statement.
%%%%%%%%%%%%%%%%%%%%%%%%%%%%%%%%%%%%%%%%%%%%%%%
\noindent \textbf{{\footnotesize Journal of Algebraic Systems\\ Vol. XX, No XX, (201X), pp XX-XX}}\\[1.00in]
%%%%%%%%%%%%%%%%%%%%%%%%%%%%%%%%%%%%%%%%%%%%%%%

%%%%%%%%%%%%%%%%%%%%%%%%%%%%%%%%%%%%%%%%%%%%%%%%%%%%%%%%%%%%%%%%%%%%%
% Insert title of your article. Note: \title[short title]{full title}
%%%%%%%%%%%%%%%%%%%%%%%%%%%%%%%%%%%%%%%%%%%%%%%%%%%%%%%%%%%%%%%%%%%%%
\title[On Relative Central Extensions and Covering Pairs]{On Relative Central Extensions and Covering Pairs}
%%%%%%%%%%%%%%%%%%%%%%%%%%%%%%%%%%%%%%
% Author's name must be inserted here
%%%%%%%%%%%%%%%%%%%%%%%%%%%%%%%%%%%%%%
\author[Pourmirzaei, Hassanzadeh and Mashayekhy]{A. Pourmirzaei$^{*}$, M. Hassanzadeh and B. Mashayekhy}
%%%%%%%%%%%%%%%%%%%%%%%%
\thanks{{\scriptsize
\hskip -0.4 true cm MSC(2010): Primary: 20E34; Secondary: 20E22, 20F05
Received: 2 March 2015, Accepted: 26 February 2016.\\
$*$Corresponding author }}

%%%%%%%%%%%%%%%%%%%%%%%%%%%%%%%%%%%%%%%%%%%
\begin{abstract}
Let $(G,N)$ be a pair of groups. In this article, first we construct a relative central extension for the pair $(G,N)$ such that special types of covering pair of $(G,N)$ are homomorphic image of it. Second, we show that every perfect pair admits at least one covering pair. Finally, among extending some properties of perfect groups to perfect pairs, we characterize covering pairs of a perfect pair $(G,N)$ under some extra assumptions.
\end{abstract}

%%%%%%%%%%%%%%%%%%%%%%%
\maketitle
%%%%%%%%%%%%%%%%%%%%%%%

In this paper the well-known notion of relative central extension of a pair of groups is used. By a pair of groups we mean a group $G$ and a normal subgroup $N$ and this is denoted by $(G,N)$. Let $M$ be another group on which an action of $G$ is given. The $G$-commutator subgroup of $M$ is defined by the subgroup $[M,G]$ of $M$ generated by all the $G$-commutators $$[m,g]=m^gm^{-1},$$ in which $g\in G$, $m\in M$ and $m^g$ is the action of $g$ on $m$. Ellis \cite{El} defined the $G$-center of $M$ to be the subgroup $$Z(M,G)=\{m\in M|m^g=m, \forall g \in G \}$$
Now we recall the definition of relative central extension of a pair of groups.
%-------------------------------------------------------------------%
\begin{definition}
(\cite{Ell})
Let $(G,N)$ be a pair of groups. A relative
central extension of the pair $(G,N)$ consists of a group
homomorphism $\sigma : M \rightarrow G$, together with an action
of $G$ on
$M$ such that \\
$(i)  \sigma(M)=N$; \\
$(ii) \sigma(m^g)=g^{-1}\sigma(m)g$, for all $g \in G$, $m \in M$;\\
$(iii) m'^{\sigma(m)}= m^{-1}m'm$, for all $m, m' \in M$;\\
$(iv) Ker(\sigma) \subseteq Z(M,G)$,\\
\end{definition}
%-----------------------------------------------------------------------------------------%
Note that for every relative central extension $\sigma:M\rightarrow G$ of a pair of groups $(G,N)$, we have $Inn(M)\subseteq Im \sigma$, by the property $(iii)$ in Definition 0.1. Moreover, every relative central extension $\sigma:M\rightarrow G$ yields an exact sequence $$1\rightarrow ker\sigma \rightarrow M \stackrel{\sigma}{\rightarrow}N\rightarrow 1.$$ Conversely, for every exact sequence $$1\rightarrow ker\sigma \rightarrow M' \stackrel{\sigma}{\rightarrow}N\rightarrow 1$$ with properties $(i)$-$(iv)$, $\rho:M'\rightarrow G$ is a relative central extension. For this reason we do not differ between a relative central extension and its associated exact sequence.
%------------------------------------------------------------------------------------------%

Let $\sigma : M \rightarrow G$, $\sigma' : M' \rightarrow G$ be two relative central extensions of $(G,N)$. In 1998 Eliss \cite{Ell} introduced a morphism between these relative central extensions which is a group homomorphism $\varphi:M\rightarrow M'$ satisfying $\sigma'\varphi(m)=\sigma(m)$ and $\varphi(m^{g})=(\varphi(m))^{g}$ for all $g\in G$, $m\in M$. In particular, if $\varphi$ is a surjective homomorphism, then $\sigma'$ is called a homomorphic image of $\sigma$.
The resulting category is the category of relative central extensions which Eliss denoted it by $\mathcal{R}\mathcal{C}\mathcal{E}$ $(G,N)$.
%-------------------------------------------------------------------------------------------%

Now we recall the definition of a covering pair.
\begin{definition}
(\cite{Ell})
A relative central extension $\sigma: M^{*}\rightarrow G$ of the pair $(G,N)$ is called a covering pair for $(G,N)$ if there exists a subgroup $A$ of $M^{*}$ such that\\
$(i)  A\subseteq Z(M^{*},G)\cap [M^{*},G]$; \\
$(ii) A\cong M(G,N)$;\\
$(iii) N\cong M^{*}/A$,\\
where $M(G,N)$ is the Schur multiplier of the pair $(G,N)$.
\end{definition}

A covering pair $\sigma:G^{*}\rightarrow G$ of the pair $(G,G)$ coincides with the usual notion of a covering group $G^{*}$ of the group $G$. In 1998 Ellis \cite{Ell} showed that every pair of finite groups admits at least one covering pair.
%---------------------------------------------------------------------------------------------%

In this paper we consider $P$ as a perfect group that satisfies in maximal condition which means $\phi(P)\neq P$ such that $\phi(P)$ is the Frattini subgroup of $P$. We mention that if $\sigma: P\rightarrow G$ is a covering pair of $(G,N)$, then we have $P'=[P,G]$ and therefore $$N=\sigma(P)=\sigma([P,G])=[\sigma(P),G]=[N,G].$$ Let $(G,N)$ be a pair of groups with free presentation $G\cong F/R$ and $N\cong S/R$ for a normal subgroup $S$ of $F$, where $F$ is the free group on the set $G$. In Section 1, we introduce a relative central extension $\delta: S/[R,F]\rightarrow G$ which has an important role throughout the paper. As a consequence, we show that for every covering pair $\sigma:P\rightarrow G$, $P$ is a homomorphic image of $S/[R,F]$. We remark that in 2007 Salemkar, Moghaddam and Chiti {\rm }\cite[Lemma 2.5]{Sa} claimed that there exists an epimorphism from $\delta$ to any relative central extension without any condition. We present a counterexample to show that their claim is not correct and hence one of their main results {\rm }\cite[Theorem 2.6]{Sa} is not valid.

In 1978 Loday \cite{L} extended the notion of perfect group to perfect pair in the sense that a pair of groups $(G,N)$ is perfect if $[G,N]=N$. Moreover he proved that $(G,N)$ is a perfect pair of groups if and only if $\mathcal{R}\mathcal{C}\mathcal{E}$ $(G,N)$ has a universal object. To prove this he used some cohomological methods. In Section 2, by restriction of $\delta$ to $[S,F]/[R,F]$ we obtain a universal object in $\mathcal{R}\mathcal{C}\mathcal{E}$ $(G,N)$ when $(G,N)$ is perfect. This is also a covering pair of $(G,N)$. It is worth to mention that we use only presentation methods instead of cohomological methods which seems easier.\\ In sequel we extend a result of Schur on covering groups to covering pairs.
%---------------------------------------------------------------------------------------%
\section{Some Results for Relative Central Extensions}
%----------------------------------------------------------------------------------------%
Let $(G,N)$ be a pair of groups with a free presentation $1\rightarrow R\rightarrow F \stackrel{\pi}{\rightarrow}G\rightarrow1$ such that $N\cong S/R$ for a normal subgroup $S$ of $F$.
Define the group homomorphism
\begin{eqnarray}
 %\nonumber to remove numbering (before each equation)
\delta : \frac{S}{[R,F]} & \rightarrow & G, \\
s[R,F]& \mapsto & \pi(s). \ \  \ \
\end{eqnarray}
It is straightforward to check that $\delta$ is a
relative central extension by the following action
\begin{eqnarray}
\bar{\delta} : \frac{S}{[R,F]}\times\frac{F}{R} & \rightarrow & \frac{S}{[R,F]}, \\
(s[R,F],fR)& \mapsto & s^f[R,F]. \ \  \ \  \ \ \ \
\end{eqnarray}

In the following theorem we present a relation between the above relative central extension to anyone of the pair $(G,N)$, especially to any covering pair $\sigma:M\rightarrow G$ of $(G,N)$.
%----------------------------------------------------------------------------------------------
\begin{theorem}\label{t1}
Let $G\cong F/R$, where $F$ is the free group on the set $G$ and $N$ be a normal subgroup of $G$ with $N\cong S/R$ for a normal subgroup $S$ of $F$. If $\sigma: M\rightarrow G$ is a relative central extension of the pair $(G,N)$, then there exists a homomorphism $\beta: S/[R,F]\rightarrow M$ such that the following diagram is commutative:
$$
\begin{array}{ccccccccc}
1&\rightarrow&\frac{R}{[R,F]}&\rightarrow&\frac{S}{[R,F]}& \stackrel{\delta}{\rightarrow}&N&\rightarrow&1 \\
& & \beta| \downarrow  & & \beta \downarrow& & \parallel& & \\
1&\rightarrow&A&\rightarrow&M& \stackrel{\sigma}{\rightarrow}&N&\rightarrow&1 ,\\
\end{array}
$$
where $\delta$ is the relative central extension $(4)$. In particular, if M is a perfect group with $\phi(M)\neq M$, then the above homomorphism $\beta$ is an epimorphism.
\end{theorem}
\begin{proof}
Consider the map $f:G\rightarrow N$ by the rule $f(g)=g$ if $g\in N$, otherwise $f(g)=1$. Now $f$ induces a homomorphism $\theta: F\rightarrow N$. Using the projective property of free groups, there is a homomorphism $\alpha:F\rightarrow M$ such that $\sigma\alpha=\theta$. By the restriction of $\theta$ and $\alpha$ to $S$ we have the following commutative diagram:
$$
\begin{array}{ccccc}
 & &S& &\\
 &\alpha| \swarrow &\downarrow \theta|& & \\
M&\stackrel{\sigma}\longrightarrow &N&\longrightarrow&1. \\
\end{array}
$$
For $x\in F$ and $r\in R$, $\alpha([r,x])=[\alpha(r), \alpha(x)]=1$ since $\alpha(r)\in ker\sigma \leq Z(M,G)\leq Z(M)$. Therefore $[R,F]\leq ker \alpha|$ which implies the existence of $$\beta: \frac{S}{[R,F]}\rightarrow M.$$ Clearly $\beta(R/[R,F])\leq ker \sigma,$ which gives the restriction of $\beta$ to $R/[R,F]$, $\beta|$. Now let $A=ker\sigma\leq Z(M,G)$ and $\phi(M)\neq M$. For every $m\in M$, there exists $\bar{s}\in S/[R,F]$ such that $\sigma(m)=\delta(\bar{s})=\sigma\beta(\bar{s})$. Hence $m=\beta(\bar{s})a$, for some $a\in ker\sigma$. Thus $$M=\langle\beta(\bar{s}), a \ \ | \ \ \bar{s}\in \frac{S}{[R,F]},a\in A\rangle.$$ Since $M=M'$ so $$A\leq Z(M,G)\cap M' \leq Z(M)\cap M' \leq \phi(M).$$ Hence $$M=\langle\beta(\bar{s})\ \ | \ \ \bar{s}\in \frac{S}{[R,F]}\rangle$$ which implies that $\beta$ is an epimorphism.
\end{proof}
\begin{corollary}
With the assumptions and notations of the previous theorem, for every covering pair $\sigma: P\rightarrow G$ of $(G,N)$, $P$ is a homomorphic image of $S/[R,F]$.
\end{corollary}
%-------------------------------------------------------------------------------------------------------------------

Note that in Theorem \ref{t1} if $$\beta(\bar{s}^{g})=\beta(\bar{s})^{g}$$ for every $\bar{s}\in S/[R,F]$ and $g\in G$ (in other words $\beta$ preserves the action), then we can say that every covering pair $\sigma: P\rightarrow G$ is a homomorphic image of $\delta: S/[R,F]\rightarrow G$.
%--------------------------------------------------------------------------------------------------------------------
\begin{remark} In 2007 Salemkar, Moghaddam and Chiti {\rm }\cite[Lemma 2.5]{Sa} claimed that in the above theorem $\beta$ is always an epimorphism for any free presentation $G\cong F/R$ and $N\cong S/R$ without any condition. The following counterexample shows that their claim is not true and hence one of the main results of their paper {\rm }\cite[Theorem 2.6]{Sa} is not valid.
\end{remark}
%---------------------------------------------------------------------------------------------------------------
\begin{example} Put $G=N=\textbf{Z}\oplus ...\oplus\textbf{Z}$ a free abelian group of finite rank with the free presentation $G\cong F/F'\cong N$, where $F'$ is the commutator subgroup of the free group $F$. Consider the following commutative diagram:
$$
\begin{array}{ccccccccc}
1&\longrightarrow&\frac{F'}{[F',F]}& \stackrel{\subseteq}{\longrightarrow}&\frac{F}{[F',F]}& \stackrel{\delta}{\longrightarrow}&G=N=\textbf{Z}\oplus ...\oplus\textbf{Z}&\longrightarrow&1 \\
& & \beta| \downarrow  & & \beta \downarrow& & \parallel& & \\
1&\longrightarrow&A=\textbf{Z}&\stackrel{\alpha}{\longrightarrow}&M=G\oplus \textbf{Z}& \stackrel{\sigma}{\longrightarrow}&G=N=\textbf{Z}\oplus ...\oplus\textbf{Z}&\longrightarrow&1 ,\\
\end{array}
$$
where $\alpha(x)=(1,x)$ and $\sigma(g,x)=g$, for all $x\in A$, $g\in G$. Then we have $$\beta(\frac{F'}{[F',F]})=[\beta(\frac{F}{[F',F]}),\beta(\frac{F}{[F',F]})]\leq M'=1.$$ Hence $\beta$ is not an epimorphism.
\end{example}
%----------------------------------------------------------------------------------------------------------------
At the end of this section by a result of Ellis {\rm }\cite[Corollary 1.2]{Ell} on the Schur multiplier of a pair $(G,N)$ of finite nilpotent groups we present a covering pair of the pair $(G,N)$.
%-----------------------------------------------------------------------------------------%
\begin{theorem}
Let $(G,N)$ be a pair of finite nilpotent groups. Assume that $G=\prod^{k}_{i=1}S_{i}$, where $S_{i}$ is the Sylow $p_{i}$-subgroup of $G$ and $\sigma_{i}:M_{i}\rightarrow S_{i}$ for each  $i=1,...,k$, is an arbitrary covering pair of the pair $(S_{i},S_{i}\cap N)$. Then $\sigma:\prod^{k}_{i=1} M_{i}\rightarrow G$ defined by $\sigma(\{m_{i}\}_{i=1}^{k})=\{\sigma_{i}(m_{i})\}_{i=1}^{k}$ is a covering pair of $(G,N)$.
\end{theorem}
%--------------------------------------------------------------------------------------------%
\begin{proof} It is readily seen that $\sigma$ is a relative central extension of $(G,N)$. By {\rm }\cite[Corollary 1.2]{Ell} $M(G,N)\cong \Pi_{i=1}^{k}M(S_{i},S_{i}\cap N)$. hence the proof is straightforward.
\end{proof}
%---------------------------------------------------------------------------------------%
\section{Some Properties of Covering Pairs}
The aim of this section is to introduce a covering pair for a perfect pair of groups $(G,N)$.
\begin{theorem}
Let $(G,N)$ be a perfect pair of groups. Then for every covering pair $\sigma:M\rightarrow G$, we have $[M,G]=M$.
\end{theorem}
%--------------------------------------------------------------------------------------------%
\begin{proof} Since $\sigma:M\rightarrow G$ is a covering pair, we have $M/ker\sigma \cong N$ and $ker \sigma \leq [M,G]$. Hence $$\frac{M/ker\sigma}{[M,G]/ker\sigma}\cong\frac{N}{[N,G]}.$$ The result is a consequence of $(G,N)$ being perfect.
\end{proof}
%---------------------------------------------------------------------------------------------%

 It is known that if $G\cong F/R$ is a group with a covering group $G^*$, then there exists a normal subgroup $S$ of $F$ such that $R/[R,F]=(F'\cap R)/[R,F]\times S/[R,F]$, and we have an isomorphism $F/S\cong G^*$ {\rm }\cite[Theorem 2.4.6(iv)(b)]{Ka}. In the following theorem we intend to extend this result to pairs of groups.
\begin{theorem}
Let $(G,N)$ be a pair of groups with $G\cong F/R$ and $N\cong S/R$, where $F$ is the free group on the set $G$ and $S\unlhd F$. Then for every relative central extension $\sigma: P\rightarrow G$ with $A\cong ker\sigma$ there is a normal subgroup $T$ of $F$ such that\\
$$\frac{R}{[R,F]}=\frac{R\cap[S,F]}{[R,F]}\times \frac{T}{[R,F]}.$$\\
Moreover, there exists an isomorphism $S/T\cong P$ which carries $R/T$ onto $A$.
\end{theorem}
%---------------------------------------------------------------------------------------------%
\begin{proof} Let $\pi: F\rightarrow G $ be a surjective homomorphism with $R=ker \pi$. As the proof of Theorem \ref{t1}, there exists a homomorphism $\psi:S\rightarrow P$ which induces the epimorphism $\beta:S/[R,F]\rightarrow P$. Thus $\psi$ is in fact an epimorphism. By setting $T=ker \psi$,  we have $S/T \cong P$. Clearly $\psi(R)=A$ and by Theorem \ref{t1}, $\psi([R,F])=1$ so $[R,F]\leq T$ which gives the induced epimorphism $\bar{\psi}:R/[R,F]\rightarrow A$. Considering $$\bar{\psi}|:\frac{R\cap[S,F]}{[R,F]}\rightarrow A,$$ we claim that $\bar{\psi}|$ is an isomorphism. To prove this, let $a\in A$. By $\psi(R)=A$ there exists $x\in R$ such that $a=\psi(x)$. On the other hand for every $p\in P$ and $g\in G$,
$$[p,g]=[\beta(\bar{s}),g]=\beta(\bar{s}^{-1}) \beta(\bar{s}^{g}).$$
Easily it can be seen that

\[\beta ({\bar s^g}) = \left\{ {\begin{array}{*{20}{c}}
   {\beta {{(\bar s)}^g}} & {ifg \in N}  \\
   {\beta (\bar s)\beta (\bar r)} & {ifg \notin N}  \\

 \end{array} } \right.\]
for some $\bar{r}\in R/[R,F]$.\\
Thus by the proof of Theorem \ref{t1},$$P=\langle \beta(\bar{s})\   \ |\    \ \bar{s}\in \frac{S}{[R,F]} \setminus \frac{R}{[R,F]}\rangle.$$ Since $P$ is perfect and so $(G,N)$ is perfect, we have $[S,F]R=S$. On the other hand $A=\psi(R)\leq \phi(P)$. Therefore $[P,G]\leq \psi([S,F])$. This fact and $A\leq [M,G]$ imply that $a=\psi(y)$ for some $y\in [S,F]$. Thus $\psi(x)=a=\psi(y)$, hence $x^{-1}y\in ker\psi\leq R$, so that $y\in R\cap[S,F]$. Since $A\cong M(G,N)\cong (R\cap [S,F])/[R,F]$ and $A$ is finite, $\bar{\psi}|$ is an isomorphism. By surjectivity of  $\bar{\psi}$, for every $\bar{r}\in R/[R,F]$ there exists $\bar{x}\in (R\cap [S,F])/[R,F]$ such that $\bar{\psi}(\bar{x})$=$\bar{\psi}(\bar{r})$, thus $\bar{x}^{-1}\bar{r}\in ker \psi$. Consequently $$\frac{R}{[R,F]}=ker \bar{\psi}\  \frac{R\cap[S,F]}{[R,F]}.$$ Let $\bar{x}\in ker \bar{\psi}\cap (R\cap [S,F])/[R,F]$. Then $\bar{x}\in ker \bar{\psi}|=1$, so we have the following isomorphism  $$\frac{R}{[R,F]}\cong ker \bar{\psi}\times \frac{R\cap[S,F]}{[R,F]}= \frac{T}{[R,F]}\times \frac{R\cap[S,F]}{[R,F]}.$$
\end{proof}
%------------------------------------------------------------------------------------------%
\begin{theorem}
Let $\sigma_{i}: M_{i}\rightarrow G$, $i=1,2$ be two covering pairs of a finite pair $(G,N)$. If $M_{i}=M'_{i}$ and $\phi(M_{i})\neq M_{i}$ for $i=1,2$, then \\
$(i)M_{1}\cong M_{2}$;\\
$(ii)M_{1}/Z(M_{1},G)\cong M_{2}/Z(M_{2},G)$;\\
$(iii)Z(M_{1},G)/ker\sigma_{1}\cong Z(M_{2},G)/ker\sigma_{2}$.
\end{theorem}
%--------------------------------------------------------------------------------------------%
\begin{proof}
$(i)$ Let $\sigma: M^{*}\rightarrow G$ be a fixed arbitrary covering pair with the following relative central extension $$1\rightarrow A\rightarrow M^{*}\rightarrow N\rightarrow 1$$ such that $A\subseteq Z(M^{*},G)\cap [M^{*},G]$ and $A\cong M(G,N)$. Since $[N,G]=N$ and $[M^{*},G]=M^{*}$ so $\sigma: [M^{*},G]\rightarrow [N,G]$ is an epimorphism. Therefore
\[\begin{array}{*{20}{c}}
   {|[M^{*},G]|} &  =  & {|\ker \sigma |} & {|[N,G]|}  \\
   {} &  =  & {|A|} & {|[N,G]|}  \\
   {} &  =  & {|M(G,N)|} & {|[N,G]|} . \\
   {} & {} & {} & {}  \\

 \end{array} \]
Suppose that $G\cong F/R$ such that $F$ is the free group on the set $G$ with $N\cong S/R$. By Theorem \ref{t1}, we have the following diagram
$$
\begin{array}{ccccccccc}
1&\rightarrow&\frac{R}{[R,F]}&\rightarrow&\frac{S}{[R,F]}& \stackrel{\delta}{\rightarrow}&N&\rightarrow&1 \\
 & & \beta| \downarrow& & \beta \downarrow& & \parallel & & \\
1&\rightarrow&A&\rightarrow&M^{*}& \stackrel{\sigma}{\rightarrow}&N&\rightarrow&1. \\
\end{array}
$$
Since $[S/[R,F],G]=[S,F]/[R,F]$ and $\delta_{|}: [S,F]/[R,F]\rightarrow [N,G]$ is an epimorphism, thus $$|\frac{[S,F]}{[R,F]}|=|M(G,N)| \\ |[N,G]|.$$
Therefore  $$|M^{*}|=|[M^{*},G]|=|\frac{[S,F]}{[R,F]}|.$$ If $\bar{s}\in S/[R,F]$ and $g\in G$,

\[\beta ([\bar s,g]) = \beta {(\bar s)^{ - 1}}\beta ({\bar s^g}) = \left\{ {\begin{array}{*{20}{c}}
   {\beta {{(\bar s)}^{ - 1}}\beta {{(\bar s)}^g} = [\beta (\bar s),g]} & {ifg \in N}  \\
   {\beta {{(\bar s)}^{ - 1}}\beta (\bar s) = 1} & {ifg \notin N}  \\

 \end{array} } \right.\]
which yields that $\beta([S,F]/[R,F])\leq [M^{*},G].$ By Theorem 2.1, $M^{*}=[M^{*},G]\leq \beta([S,F]/[R,F])$ which implies that $M^{*}\cong [S,F]/[R,F]$.\\
$(ii)$ By Theorem 2.2 there exists a normal subgroup $T$ of $F$ with $M^{*}\cong S/T.$ This yields naturally an action of $G$ on $S/T$ which implies $Z(S/T,G)=Z(M^{*},G).$ Put $Z(S/T,G)=L/T.$ Let $x\in Z(S/[R,F],G)$. Then $[x,g]\in [R,F]\leq T$ and so $Z(S/[R,F],G)\leq L/[R,F]$. To prove the reverse containment, assume that $x[R,F]\in L/[R,F]$, then $[x,g]\in T\leq R\cap [S,F]\cap T=[R,F]$,  hence $$\frac{L}{[R,F]}=Z(\frac{S}{[R,F]},G).$$ Consequently, it may be inferred that $$\frac{M^{*}}{Z(M^{*},G)}\cong \frac{S/T}{L/T}\cong\frac{S}{L}.$$ The desired assertion is now a consequence of the fact that $S/L$ is determined by the presentation of $G$ and $N$.\\
$(iii)$ Owing to Theorem \ref{t1} we have $$\frac{Z(M^{*},G)}{A}\cong\frac{L/T}{R/T}\cong\frac{L}{R}.$$ The result is now established.
\end{proof}
%---------------------------------------------------------------------------------------%
Now, we are in a position to state and prove one of the main results of this section.
%----------------------------------------------------------------------------------------%
\begin{theorem}
Let $(G,N)$ be a perfect pair of groups with a free presentation $G\cong F/R$ and $N\cong S/R$. Then\\
$(i)\ \ \delta:[S,F]/[R,F]\rightarrow G$ by $\delta(x[R,F])=xR$, is a covering pair of $(G,N)$\\
$(ii)$ The relative central extension $\delta:[S,F]/[R,F]\rightarrow G$ is the universal object in the category $\mathcal{R}\mathcal{C}\mathcal{E}$ $(G,N)$.
\end{theorem}
%--------------------------------------------------------------------------------------------%
\begin{proof}
$(i)$ Consider the following action of $G$ on $[S,F]/[R,F]$
$$\frac{[S,F]}{[R,F]}\times G \rightarrow \frac{[S,F]}{[R,F]}$$
$$([s,f][R,F],g)\mapsto ([s,f][R,F])^{g}=[s^{t},f^{t}][R,F],$$ where $\pi:F\rightarrow G$ is the epimorphism with $ker\pi= R$ and $\pi(t)=g$. By the assumption $G=NQ$, let $Q\cong T/R$. Since $(G,N)$ is a perfect pair, $[N,G]=N$ and therefore $([S,F]R)/R=S/R$. It can be shown that $\delta:[S,F]/[R,F]\rightarrow G$ is a relative central extension of the pair $(G,N)$. On the other hand, since $$ker \delta= (R\cap [S,F])/[R,F]\cong M(G,N),$$ we have $$\frac{[S,F]/[R,F]}{ker \delta}\cong \frac{[S,F]}{R\cap [S,F]}\cong \frac{[S,F]R}{R}=\frac{S}{R}\cong N.$$ It is enough to show that $$ker \delta \leq [\frac{[S,F]}{[R,F]},G].$$ For this we show that $[[S,F]/[R,F],G]=[S,F]/[R,F]$. Clearly $$[[S,F]/[R,F],G]\leq [S,F]/[R,F].$$ So we only need to show the reverse containment. We know that $$[\frac{[S,F]}{[R,F]},G]=\langle[x,g]| x\in \frac{[S,F]}{[R,F]} , g\in G\rangle$$ such that
$$[x,g]=x^{-1}x^{g}=[s,f][R,F]^{-1}([s,f][R,F])^{g}=[s,f]^{-1}[s,f]^{t}[R,F],$$ for some $s\in S$ and $f\in F$, where $\pi(t)=g$. Since $S=[S,F]R$, for every $s\in S$ there exist $s'\in S$, $l\in F$ and $r\in R$ such that $$[s,f]=[[s',l]r,f]=[[s',l],f][[s',l],f,r][r,f].$$ Hence $[S,F]/[R,F]\leq [[S,F]/[R,F],G]$.\\
$(ii)$ Let $\sigma:M\rightarrow G$ be a relative central extension of the pair $(G,N)$, and $F$ be the free group on the set $G$ and $S\unlhd F$ such that $G\cong F/R$ and $N\cong S/R$. Then consider $\delta:[S,F]/[R,F]\rightarrow G$ as mentioned in the previous part. By Theorem\ref{t1}, there exists a homomorphism $$\varphi: [S,F]/[R,F]\rightarrow M$$ such that $\varphi=\beta|_{[S,F]/[R,F]}$. Now for every $g\in G$ and $\bar{x}\in [S,F]/[R,F]$,

\[\varphi ({{\bar x}^g}) = \beta ({{\bar x}^g}) = \left\{ {\begin{array}{*{20}{c}}
   {\beta {{\left( {\bar x} \right)}^g} = \varphi {{\left( {\bar x} \right)}^g}} & {ifg \in N}  \\
   {\beta (\bar x) = \varphi \left( {\bar x} \right)} & {ifg \notin N}  \\

 \end{array} } \right.\]
We have to prove $\varphi$ preserves the action. To do this it is enough to show that $\varphi(\bar{x})^g= \varphi(\bar{x})$ for every $g\in G$ and $\bar{x}\in [S,F]/[R,F]$. Define
\begin{eqnarray*}
\ f:\frac{S}{[R,F]} & \rightarrow & \ M \\
\bar{s}& \mapsto & \beta(\bar{s})^{g} \beta(\bar{s})^{-1}.
\end{eqnarray*}
For every $g\in G$ and $\bar{x}\in [S,F]/[R,F]$, $f(\bar{s})\in Z(G,M)$ so
\[f({{\bar s}^g}) = \beta {({{\bar s}^g})^g}\beta {({{\bar s}^g})^{ - 1}} = \left\{ {\begin{array}{*{20}{c}}
   {{{(\beta {{(\bar s)}^g})}^g}{{(\beta {{(\bar s)}^{ - 1}})}^g} = f{{(\bar s)}^g}} & {ifg \in N}  \\
   {\beta {{(\bar s)}^g}\beta {{(\bar s)}^{ - 1}}} & {ifg \notin N}  \\

 \end{array} } \right.\]
Thus,
\[f([\bar s,g]) = f({\bar s^{ - 1}}{\bar s^g}) = \left\{ {\begin{array}{*{20}{c}}
   {f{{(\bar s)}^{ - 1}}f{{(\bar s)}^g} = [f(\bar s),g] = 1} & {ifg \in N}  \\
   {f{{(\bar s)}^{ - 1}}f(\bar s) = 1} & {ifg \notin N}  \\

 \end{array} } \right.\]
By part $(i)$ $$f(\frac{[S,F]}{[R,F]})=f([\frac{[S,F]}{[R,F]},G])=1.$$ Therefore $$\varphi(\bar{x}^{g})=\beta(\bar{x}^{g})=\beta(\bar{x})^{g}=\varphi(\bar{x})^{g}.$$
We claim that $\varphi$ is unique. By contrary, let $\varphi_{i}:[S,F]/[R,F]\rightarrow M$, $i=1,2$, be homomorphisms such that $$\sigma \varphi_{1}(\bar{s})=\delta(\bar{s})=\sigma \varphi_{2}(\bar{s}).$$ Define
 \begin{eqnarray*}
\psi:\frac{[S,F]}{[R,F]} & \rightarrow & M \\
\bar{x}& \mapsto & \varphi_{1}(\bar{x}) {\varphi_{2}}^{-1}(\bar{x}).
\end{eqnarray*}
If $g\in G$ and $\bar{x}\in [S,F]/[R,F]$,$$\psi({\bar{x}}^{g})=\varphi_{1}({\bar{x}}^{g}){\varphi_{2}}^{-1}({\bar{x}}^{g})={\varphi_{1}(\bar{x})}^{g}({\varphi_{2}}^{-1}(\bar{x}))^{g}=
(\varphi_{1}(\bar{x}){\varphi_{2}}^{-1}(\bar{x}))^{g}=\psi(\bar{x})^g.$$ Consequently $\psi([\bar{x},g])=[\psi(\bar{x}),g]=1.$ Now by $(i)$ we have\\ $[[S,F]/[R,F],G]=[S,F]/[R,F]$ which implies that $\psi=1$. Hence the results holds.
\end{proof}
%----------------------------------------------------------------------------------------%

Note that if $(G,N)$ is a pair of groups with a free presentation $G\cong F/R$ and $N\cong S/R$, then $[S,F]/[R,F]$ is independent of the choice of the free presentation of $(G,N)$.
%---------------------------------------------------------------------------------------------------%

 The following theorem states a necessary and sufficient condition for a pair of groups to be perfect. It should be noted that Loday \cite{L} proved the following result using homological method but our proof is different and seems elementary.
%---------------------------------------------------------------------------------------------%
\begin{theorem}
A pair of groups $(G,N)$ is perfect pair if and only if the category $\mathcal{R}\mathcal{C}\mathcal{E}$ $(G,N)$ has a universal relative central extension.
\end{theorem}
%--------------------------------------------------------------------------------------------%
\begin{proof}Necessity is Theorem 2.4. For sufficiency, let $$1\rightarrow A\rightarrow M\stackrel{\rho}{\rightarrow} G\rightarrow 1$$ be a universal relative central extension of the pair $(G,N)$. Consider the following relative central extension $$1\rightarrow A\times\frac{N}{[N,G]}\rightarrow M\times\frac{N}{[N,G]}\stackrel{\psi}{\rightarrow}N\rightarrow 1, $$ where $\psi(m,\bar{n})=\rho(m)$, $m\in M$, $\bar{n}\in N/[N,G]$. Define homomorphisms $$\theta_{i}:M\rightarrow M\times\frac{N}{[N,G]},$$ $i=1,2$, by $\theta_{1}(m)=(m,1)$ and $\theta_{2}(m)=(m,\overline{\rho(m)})$. Then we have $\psi\theta_{i}=\rho$, $i=1,2$, where $\theta_{1}=\theta_{2}$. Consequently, $N=[N,G]$ and hence $(G,N)$ is a perfect pair.
\end{proof}
%--------------------------------------------------------------------------------------------%

 Finally, we intend to consider a covering pair as a pair of groups. If $\sigma:M\rightarrow G$ is a covering pair of $(G,N)$, then there exists a subgroup $A$ of $M$ such that $A\subseteq Z(M,G)\cap [M,G]$, $ A\cong M(G,N)$ and $M/A\cong N$. Now, we can consider $(M,A)$ as a covering pair of $(G,N)$. It is known that any two covering groups of a finite group $G$ are isoclinic. Moreover, if $G$ is a finite perfect pair of groups, then a covering group of $G$ is also perfect. We intend to investigate these facts for a covering pair as our new point of view of a covering pair of groups. In this way, the notion of isoclinism for the pair of groups is needed which we review it of \cite{Sal}.
%----------------------------------------------------------------------------------------%
\begin{definition}
Let $(G,N)$ and $(H,K)$ be two pairs of groups. The pairs $(G,N)$ and $(H,K)$ are said to be isoclinic if there exists isomorphisms $\epsilon:G/Z(G,N)\rightarrow H/Z(H,K)$ and $\eta:[G,N]\rightarrow [H,K]$ such that $\epsilon(N/Z(G,N))=K/Z(H,K)$ and $\eta([g,n])=[h,k]$ whenever $\epsilon(gZ(G,N))=hZ(H,K)$ and $\eta(nZ(G,N))=kZ(H,K)$. This concept is denoted as $(G,N)\sim (H,K)$.
\end{definition}
%----------------------------------------------------------------------------------------%

It is a well-known fact that all covering groups of a given finite group are mutually isoclinic (see \cite{Jo}). The following lemma helps us to prove this fact for some special types of covering pairs for arbitrary pair of groups.
%----------------------------------------------------------------------------------------%
\begin{lemma}
Let (G,N) and (H,K) be two pairs of groups. If $\theta:G\rightarrow H$ is an epimorphism such that $\theta(N)=K$ and $ker\theta\cap N=1$, then $(G,N)\sim (H,K).$
\end{lemma}
%---------------------------------------------------------------------------------------%
\begin{proof} Define $\epsilon:G/Z(G,N)\rightarrow H/Z(H,K)$ such that $\epsilon(gZ(G,N))=\theta(h)Z(H,K)$ and $\varphi:[N,G]\rightarrow [K,H]$ by $\varphi([n,g])=[\theta(n),\theta(g)]$. We can easily see that $\epsilon$ and $\varphi$ are isomorphisms which yields that $(G,N)\sim (H,K)$.
\end{proof}
%-------------------------------------------------------------------------------------------------------------------
\begin{theorem}
Let $(G,N)$ be a pair of groups. Then every two covering pairs $(P_{i},A_{i})$, $i=1,2$, where $P_{i}=P'_{i}$ and $\phi(P_{i})\neq P_{i}$ of $(G,N)$ are isoclinic.
\end{theorem}
%--------------------------------------------------------------------------------------------%
\begin{proof} By the proof of Theorem 2.2 if $(P,A)$ is a covering pair of $(G,N)$ then there exists an epimorphism $\beta:S/[R,F]\rightarrow P$ such that $\psi(([S,F]\cap R)/[R,F])=A$. Clearly $$ker \psi \cap [\frac{S}{[R,F]},\frac{[S,F]\cap R}{[R,F]}]=1.$$ Using Lemma 2.7 we obtain $$(P,A)\sim (\frac{S}{[R,F]},\frac{[S,F]\cap R}{[R,F]}).$$
\end{proof}
%---------------------------------------------------------------------------------------------%
 The following results can be easily obtained if we use the notion of isoclinism for a perfect pair of groups. Note that these results are generalizations of similar results for perfect groups to perfect pairs ( see {\rm }\cite[Lemma 2.4, Corollary 2.5 and Proposition 2.6]{Jo}).
%---------------------------------------------------------------------------------------------%
\begin{lemma}
Let $(G,N)$ be a finite perfect pair of groups, and $(H,K)$ be any finite pair of groups which is isoclinic to $(G,N)$. Then $K$ is isomorphic to $$\frac{N\times Z(H,K)}{Z(G,N)\times(Z(H,K)\cap [H,K])}.$$
\end{lemma}
%---------------------------------------------------------------------------------------------%
\begin{corollary}
Let $(G,N)$ be a finite perfect pair of groups with trivial center. Then for any pair of groups isoclinic to $(G,N)$, say $(H,K)$, we have $K$ is isomorphic to direct product of $N$ and an abelian group.
\end{corollary}
%---------------------------------------------------------------------------------------------%
\begin{proposition}
Let $(G,N)$ be a finite perfect pair of groups, and $(H,K)$ be any pair which is isoclinic to $(G,N)$ such that $|N|=|K|$. Then $N\cong K$.
\end{proposition}
\begin{proof}The result follows from Lemma 3.9 and the assumption that $|K|=|N|.$
\end{proof}
%---------------------------------------------------------------------------------------------%

%%%%%%%%%%%%%%%%%%%%%%%%%%%%%%%%%%%%%%%%%%%
% References
%%%%%%%%%%%%%%%%%%%%%%%%%%%%%%%%%%%%%%%%%%%
\bibliographystyle{amsplain}

\begin{thebibliography}{5}

\bibitem{El}  G. Ellis, Capability, homology and central series of a pair of groups, {\em J. Algebra} (1996), 179:31--46.

\bibitem{Ell} G. Ellis, The Schur multiplier of a pair of groups, {\em Appl. Categ. Structures} {\bf 6}(1998), 355--371.

\bibitem{He} N. S. Hekster, On the structure of n-isoclinism classes of groups, {\em J. Pure and Applied Algebra} {\bf 40}(1986), 63--85.

\bibitem{Jo} M. R. Jones, J. Wiegold, Isoclinisms and covering groups, {\em Bull. Austral. Math. Soc.} (1974), 11:71–-76.

%\bibitem{Lee} C. R Leedham-Green, S. Mckay, Baer-invariant, isologism, varietal laws and homology, {\em Acta Math.} {\bf 137} (1976), 99--150.

\bibitem{L}  J. L. Loday, Cohomologie et group de Steinberg relatif, {\em J. Algebra} {\bf 54} (1978), 178--202.

\bibitem{Ka} G. Karpilovsky, The Schur Multiplier, {\em Oxford} Clarendon Press, 1987.

%\bibitem{Kay} S. Kayvanfar, M. R. R. Moghaddam, $\mathcal{V}$-perfect groups, {\em Indag. Math. N. S.} {\bf 8(4)} (1997), 537--542.

%\bibitem{Mo} M. R. R. Moghaddam,  A. R. Salemkar, Varietal isologism and covering groups, {\em Arch. Math.} {\bf 75} (2000), 8--15.

%\bibitem{Mog} M. R. R. Moghaddam, A. R. Salemkar and  H. M. Saany,  Some inequalities for the Baer-invariant of a pair of finite groups,
%{\em Indag. Math. Soc.} {\bf 81}(2006), 1--9.

\bibitem{Sa}  A. R. Salemkar, M. R. R.  Moghaddam, K. Chiti, Some properties on the Schur multiplier of a pair of groups, {\em J. Algebra} {\bf 312} (2007), 1--8.

\bibitem{Sal}  A. R. Salemkar, F. Saeedi and T. Karimi, The structure of isoclinism of pairs of groups, {\em Southeast Asian Bull. Math.} {\bf 31} (2007), 1173-1182.

\end{thebibliography}
%%%%%%%%%%%%%%%%%%%%%%%%%%%%%%%%%%%%%%%%%%%
% Please cite your relevant papers but at most total 5 papers/books.
%%%%%%%%%%%%%%%%%%%%%%%%%%%%%%%%%%%%%%%%%%%

\bigskip
\bigskip

%{\bf Received: Month xx, 200x}

{\footnotesize {\bf Azam Pourmirzaei}\; \\ {Department of
Mathematics}, {Hakim Sabzevari University, P. O. Box 96179-76487,}{ Sabzevar, Iran}\\
{\tt a.pmirzaei@gmail.com}\\

{\footnotesize {\bf Mitra Hassanzadeh}\; \\ {Department of
Mathematics}, {Department of
Mathematics}, {Ferdowsi University of Mashhad,
P.O.Box 1159-91775,} {Mashhad, Iran.}\\
{\tt mtr.hassanzadeh@gmail.com}\\

{\footnotesize {\bf Behrooz Mashayekhy}\; \\ {Department of
Mathematics}, {  Center of Excellence in Analysis on Algebraic Structures, Ferdowsi University of Mashhad,
P.O.Box 1159-91775,} {Mashhad, Iran.}\\
{\tt bmashf@um.ac.ir}\\

\end{document}